\documentclass[a4paper]{amsart}

\usepackage{amssymb}
\usepackage{amsmath}
\usepackage{amsthm}
\usepackage[all]{xy}

\newtheorem{thm}{Theorem}[section]
\newtheorem{cor}[thm]{Corollary}
\newtheorem{lem}[thm]{Lemma}

\newtheorem{prop}[thm]{Proposition}
\theoremstyle{definition}
\newtheorem{defn}[thm]{Definition}
\theoremstyle{remark}
\newtheorem{rem}[thm]{\bf Remark}

\newcommand{\amod}{A\mbox{-{\rm mod}}}
\newcommand{\akmod}{A\otimes_kK\mbox{-{\rm mod}}}
\newcommand{\dba}{\mathbf{D}^{\mathrm{b}}(A\mbox{-{\rm mod}})}
\newcommand{\dbak}{\mathbf{D}^{\mathrm{b}}(A\otimes_kK\mbox{-{\rm mod}})}
\newcommand{\kba}{\mathbf{K}^{\mathrm{b}}(A\mbox{-{\rm proj}})}
\newcommand{\kbak}{\mathbf{K}^{\mathrm{b}}(A\otimes_kK\mbox{-{\rm proj}})}
\newcommand{\opg}{\underset{g\in G}{\oplus}}
\begin{document}

\title{piecewise hereditary algebras under field extensions}

\author{Jie Li}

\date{\today}
\thanks{E-mail: lijie0$\symbol{64}$mail.ustc.edu.cn}

\begin{abstract}
	Let $A$ be a finite-dimensional $k$-algebra and $K/k$ be a finite separable field extension. We prove that $A$ is derived equivalent to a hereditary algebra if and only if so is $A\otimes_kK$.
\end{abstract}

\maketitle

\section{introduction}

Let $A$ be a finite-dimensional algebra over a field $k$ and $\amod$ be the category of finitely generated left $A$-modules. Recall that $A$ is called piecewise hereditary if there is a hereditary abelian category $H$ such that the bounded derived category $\dba$ is equivalent to $\mathbf{D}^b(H)$ as triangulated categories.

Piecewise hereditary algebras are important and well-studied in representation theory. A homological characteristic via strong global dimensions of piecewise hereditary algebras was given by Happel and Zacharia in \cite{HR}. Using this characteristic, Li proved in \cite{LL} that the piecewise hereditary property is compatible under certain skew group algebra extensions. Similarly, we prove that it is also compatible under finite separable field extensions (see Corollary \ref{phqtufe}), which is a special case of \cite[Proposition 5.1]{LJ}. 

According to \cite{HR}, a connected piecewise hereditary $k$-algebra is derived equivalent to either a hereditary $k$-algebra or a canonical $k$-algebra. Notice that the homological characteristic and hence the compatibilities mentioned above do not distinguish these two situations. In this paper, we look for a refinement. We prove that these two kinds of piecewise hereditary algebras are closed under certain base field change. More precisely, we obtain the following result.

\medskip
\noindent {\bf Main Theorem.} \textit{Let $K/k$ be a finite separable filed extension and $A$ a $k$-algebra. Then $A$ is derived equivalent to a hereditary algebra if and only if so is $A\otimes_kK$.}

\medskip
As a corollary, $A$ is derived equivalent to a canonical algebra if and only if so is $A\otimes_kK$. We also prove that $A$ is a tilted algebra if and only if so is $A\otimes_kK$.
\medskip

By \cite{JMD}, if an algebra is derived equivalent to a hereditary algebra (or a canonical algebra), then so is its skew group algebra extension under certain condition. However, the converse of this statement has not been proved. Our theorem is the field extension version of this statement with a confirmation of the converse.

Our proof of the main theorem is based on the description of hereditary triangulated categories by directing objects; see \cite[Corollary 5.5]{CR}. We are inspired by the proof of Theorem 1.1 in \cite{LZ} saying that tilted algebras are compatible under certain skew group algebra extensions.

\section{derived categories and Galois extensions}

\subsection{Derived categories and field extensions}
We fix a finite separable field extension $K/k$, and consider a finite-dimensional $k$-algebra $A$ and its scale extension $A\otimes_kK$. The algebra extension $A\rightarrow A\otimes_kK$ induces an adjoint pair $(-\otimes_kK,F)$ between finitely generated left module categories $\amod$ and $\akmod$, where 
$$-\otimes_kK\colon \amod\longrightarrow \akmod\mbox{, }
M\longmapsto M\otimes_kK\mbox{, }\forall M\in\amod$$ is the scale extension functor and $$F\colon\akmod \longrightarrow\amod$$ is the forgetful functor. 

Denoted by $\kba$ the bounded homotopy category of finitely generated projective left $A$-modules and $\dba$ the bounded derived category. Since $-\otimes_kK$ and $F$ map projective modules to projective modules, they can be extended in a natural manner to an adjoint pair between $\kba$ and $\kbak$. These two functors are also exact, so they can be extended to an adjoint pair between $\dba$ and $\dbak$. We still denote these two adjoint pairs by $(-\otimes_kK, F)$ for convenience.

Recall from \cite{NBO} that a functor $G\colon \mathcal{C}\rightarrow \mathcal{D}$ is called \emph{separable} if for any $X, Y$ in $\mathcal{C}$, there is a map $$H_{X,Y}\colon {\rm Hom}_\mathcal{D}(G(X),G(Y))\rightarrow{\rm Hom}_\mathcal{C}(X,Y)$$ such that $H_{X,Y}(G(f))=f$, for any $f\in{\rm Hom}_\mathcal{C}(X,Y)$, and $H_{X,Y}$ is natural in $X$ and $Y$. 

The functors $(-\otimes_kK, F)$ defined above for module categories, homotopy categories and derived categories are separable since the field extension $K/k$ is separable; see \cite[Example 3.6]{LJ}. Hence each object $X$ in $\amod$, $\kba$ or $\dba$ is a direct summand of $F(X\otimes_kK)$ and each $Y$ in $\akmod$, $\kbak$ or $\dbak$ is a direct summand of $F(Y)\otimes_kK$; see \cite{LJ,Ra}.
 
The following lemma due to \cite{Z} and \cite{K} will be frequently used.
\begin{lem}
	Given two objects $X$ and $Y$ in $\dba$, we have an isomorphism of vector spaces $$\mathrm{Hom}_{\dbak}(X\otimes_kK,Y\otimes_kK)\simeq\mathrm{Hom}_{\dba}(X,Y)\otimes_kK.$$ In particular, we have an isomorphism of $K$-algebras $$\mathrm{End}_{\dbak}(X\otimes_kK)\simeq\mathrm{End}_{\dba}(X)\otimes_kK,$$
	and an isomorphism of vector spaces $$\mathrm{rad}(A\otimes_kK)\simeq(\mathrm{rad}A)\otimes_kK.$$
\end{lem}

\medskip
Recall from \cite{Ri} that a complex $T$ in $\dba$ is called a \textit{tilting complex} of $A$ if, viewing $\dba$ as $\mathbf{K}^{-,\mathrm{b}}(A\mbox{-proj})$,
\begin{enumerate}
	\item $T\in\kba$;
	\item $\mathrm{Hom}_{\dba}(T,T[i])=0$ for all $i\neq 0$;
	\item $\langle T\rangle=\kba$, where $\langle T\rangle$ is the triangulated category generated by direct summands of $T$.
\end{enumerate}
Notice that the shift functor [1] is commutative with $-\otimes_kK$.

\begin{lem}\label{tilcom}
	If $T$ is a tilting complex of $A$, then $T\otimes_kK$ is a tilting complex of $A\otimes_kK$.
\end{lem}
\begin{proof}
	We check the three conditions of tilting complexes for $T\otimes_kK$. The first one is obvious and the second one is by the isomorphisms\begin{align*}
	&\mathrm{Hom}_{\dbak}(T\otimes_kK,T\otimes_kK[i])\\\simeq&\mathrm{Hom}_{\dbak}(T\otimes_kK,T[i]\otimes_kK)\\\simeq&\mathrm{Hom}_{\dba}(T,T[i])\otimes_kK=0.
	\end{align*}
	
	For the last condition. For each $Y$ in $\kbak$, $F(Y)\in\kba$ can be generated by direct summands of $T$ because $T$ is a tilting complex. Since $-\otimes_kK$ is an additive functor and maps a triangle in $\kba$ into a triangle in $\kbak$, $F(Y)\otimes_kK$ can be generated by direct summands of $T\otimes_kK$. Hence as a direct summand of $F(Y)\otimes_kK$, $Y$ can be generated by direct summands of $T\otimes_kK$.
\end{proof}

Recall that two algebras $A$ and $B$ are derived equivalent if and only if there is a tilting complex $T$ in $\dba$ such that $\mathrm{End}_{\dba}(T)^{\mathrm{op}}\simeq B$; see \cite{Ri}.

\begin{lem}\label{deot}
	If $A$ is derived equivalent to $B$, then $A\otimes_kK$ is derived equivalent to $B\otimes_kK$.
\end{lem}
\begin{proof}
	Let $T$ be a tilting complex of $A$ such that $\mathrm{End}_{\dba}(T)^{\mathrm{op}}\simeq B$. By Lemma \ref{tilcom}, $T\otimes_kK$ is a tilting complex of $A\otimes_kK$. Then $A\otimes_kK$ is derived equivalent to $B\otimes_kK$ by the following isomorphisms of algebras: $$\mathrm{End}_{\dbak}(T\otimes_kK)^{\mathrm{op}}\simeq\mathrm{End}_{\dba}(T)^{\mathrm{op}}\otimes_kK\simeq B\otimes_kK.$$
\end{proof}
\begin{rem}
	The converse of the above lemma is not true. For example, take $k$ as the field of real numbers and $K$ as the field of complex numbers. Let $A$ be $k$ and $B$ be the quaternion algebra over $k$. Then $A\otimes_kK$ and $B\otimes_kK$ are both derived equivalent to $K$, while $A$ and $B$ are not derived equivalent.
\end{rem}

\subsection{Action of the Galois group}
In this subsection we further assume that $K/k$ is a finite Galois extension. Let $G$ be the Galois group of $K/k$. For each $g$ in $G$ and $\lambda$ in $K$, denote by $^g\lambda$ the action of $g$ on $\lambda$. 

For each $g$ in $G$ and $M$ in $\akmod$, define $^gM\in\akmod$ as follows. As a set, $^gM$ is identified with $M$. The action of $A\otimes_kK$ on $^gM$ is given by $$(a\otimes\lambda)\cdot m=(a\otimes{^g\lambda})m\mbox{, }\forall a\otimes\lambda\in A\otimes_kK\mbox{, }m\in {^gM}.$$
Then $g$ induces a $k$-linear (not $K$-linear) automorphism of $\akmod$:
$$^g(-)\colon \akmod\longrightarrow \akmod\mbox{, }
M\longmapsto {^gM}.$$
For each homomorphism $f\colon M\rightarrow N$ in $\akmod$, $^gf\colon{^gM}\rightarrow{^gN}$ is given by $$(^gf)(m)=f(m)\mbox{, }\forall m\in M.$$ 

So $^g(-)$ is a functor with inverse $^{g^{-1}}(-)$. 

The functor $^g(-)$ is exact and can be extended naturally to a $k$-liner autofunctor of $\dbak$. We still denote its derived functor by $^g(-)$. The notation $X\,|\,Y$ means that $X$ is a direct summand of $Y$. 

\begin{lem}\label{fef}
	Keep the notations above, we have \\
	{\rm (1)} For each $M\in\dbak$, $F(M)\otimes_kK\simeq\opg{^gM}$ in $\dbak$.\\
	{\rm (2)} For each indecomposable object $M\in\dbak$, there is an indecomposable object $X\in\dba$ such that $X\,|\,F(M)$ and $M\,|\,X\otimes_kK$. \\
	{\rm (3)} For each indecomposable object $X\in\dba$, there is an indecomposable object $M\in\dbak$ such that $X\,|\,F(M)$ and $M\,|\,X\otimes_kK$. If there is another indecomposable object $N\in\dbak$ satisfying that $N\,|\,X\otimes_kK$, then there is a $g\in G$ such that $N\simeq{^gM}$. 
\end{lem}
\begin{proof}
	
	(1). For each $M=(M^i,d^i)\in\dbak$, we have an isomorphism of $A\otimes_kK$-modules for each $i\in\mathbb{Z}$:
	$$\phi^i\colon F(M^i)\otimes_kK\longrightarrow\opg{^g(M^i)}\mbox{, }
	m\otimes\lambda\longmapsto((^g\lambda) m)_{g\in G}\mbox{, }\forall m\in M^i\mbox{, }\lambda\in K.$$
	Since the following diagram is commutative,
	$$\xymatrix@C=10ex{F(M^i)\otimes_kK\ar[r]^{d^i\otimes\mathrm{Id}}\ar[d]^{\phi^{i}}&F(M^{i+1})\otimes_kK\ar[d]^{\phi^{i+1}}\\\opg{^g(M^i)}\ar[r]^{\mathrm{diag}((^g(d^i))_{g\in G})}&\opg{^g(M^{i+1})}}$$ we obtain an isomorphism $$\phi=(\phi^i)_{i\in\mathbb{Z}}\colon F(M)\otimes_kK\simeq\opg{^gM}$$ in $\dbak$.
	
	(2). Since the functor $F$ is separable, we have $M\,|\,F(M)\otimes_kK$ in $\dbak$. Thus there is an indecomposable direct summand $X$ of $F(M)$ such that $M\,|\,X\otimes_kK$.
	
	(3). Since the functor $-\otimes_kK$ is separable, we have $X\,|\,F(X\otimes_kK)$ in $\dba$. Thus there is an indecomposable direct summand $M$ of $X\otimes_kK$ such that $X\,|\,F(M)$. 
	
	If $N\in\dbak$ is another indecomposable object such that $N\,|\,X\otimes_kK$, then $N\,|\,F(M)\otimes_kK\simeq\opg{^gM}$ by (1). Since $\dbak$ is Krull-Schmidt, there is a $g\in G$ such that $N\simeq{^gM}$.
\end{proof}

\begin{defn}
	Given an indecomposable object $X$ in $\dba$ and an indecomposable object $M$ in $\dbak$, we say that $M$ and $X$ are \textbf{relative} if $X\,|\,F(M)$ and $M\,|\,X\otimes_kK$.
\end{defn}

\begin{lem}\label{bridge}
	Let $X$ and $Y$ be two indecomposable objects in $\dba$. Assume that $M$ and $N$ are indecomposable objects in $\dbak$ and relative with $X$ and $Y$ respectively. Then for each non-zero non-isomorphism $\phi\colon X\rightarrow Y$, there is a non-zero non-isomorphism $\psi\colon M\rightarrow {^gN}$ in $\dbak$ for some $g$ in $G$.
\end{lem}
\begin{proof}
	By Lemma \ref{fef} (3), up to isomorphism, each indecomposable direct summand of $X\otimes_kK$ and $Y\otimes_kK$ belongs to $\{^gM\,|\,g\in G\}$ and $\{^gN\,|\,g\in G\}$, respectively. Since $$\phi\otimes\mathrm{Id}\colon X\otimes_kK\longrightarrow Y\otimes_kK$$ is non-zero, there exist $h$ and $l$ in $G$ such that $\pi_{^lN}\circ(\phi\otimes\mathrm{Id})\circ\mathrm{inc}_{^hM}\neq0$, where $\mathrm{inc}_{^hM}\colon{^hM}\rightarrow X\otimes_kK$ is the embedding morphism and $\pi_{^lN}\colon X\otimes_kK\rightarrow{^lN}$ the projection morphism. 
	
	Let $$\psi:={^{h^{-1}}(\pi_{^lN}\circ(\phi\otimes\mathrm{Id})\circ\mathrm{inc}_{^hM})}\colon M\longrightarrow {^{h^{-1}l}N}.$$ Since $^{h^{-1}}(-)$ is an isomorphism, $\psi$ is non-zero. 
	
	We claim that $\psi$ is a non-isomorphism. Recall that in a Krull-Schmidt category, a morphism between two indecomposable objects is a non-isomorphism if and only if it belongs to the radical of the category; see \cite[A.3 Proposition 3.5]{ASS}. For each $$f\otimes\lambda\in\mathrm{Hom}_{\dba}(Y,X)\otimes_kK\simeq\mathrm{Hom}_{\dbak}(Y\otimes_kK,X\otimes_kK),$$ $f\circ\phi$ is a non-isomorphism, which implies that
	\begin{align*}
	(f\otimes\lambda)\circ(\phi\otimes\mathrm{Id})&\in\mathrm{rad}_{\dba}(X,X)\otimes_kK\\&=\mathrm{rad}(\mathrm{End}_{\dba}(X))\otimes_kK\\
	&\simeq\mathrm{rad}(\mathrm{End}_{\dbak}(X\otimes_kK)).
	\end{align*} 
	According to \cite[A.3 Definition 3.3]{ASS}, $\phi\otimes\mathrm{Id}\in\mathrm{rad}_{\dbak}(X\otimes_kK,Y\otimes_kK)$. By \cite[A.3 Lemma 3.4]{ASS}, we have that $$\psi\in\mathrm{rad}_{\dbak}(M,{^{h^{-1}l}N}).$$ Therefore $\psi$ is a non-isomorphism. 
\end{proof}

\section{piecewise hereditary algebras under field extensions}

\subsection{Piecewise hereditary algebras}
In this section, we recall some knowledge about piecewise hereditary algebras and investigate related properties under base field extension.

According to \cite{HR}, the \textit{strong global dimension} of a $k$-algebra $A$ is defined by $$\mbox{s.gl.}\dim A=\sup\{l(P)\,|\,0\neq P\in\kba\text{ indecomposable}\},$$ where $$l(P=(P^i,d^i))=\max\{b-a\,|\,P^b\neq 0,P^a\neq 0\}$$ is the length of $P\neq0$.

\begin{lem}\label{ph}
	Let $A$ be a $k$-algebra and $K/k$ be a finite separable field extension. Then $$\mathrm{s.gl}\dim A=\mathrm{s.gl}\dim A\otimes_kK.$$
\end{lem}
\begin{proof}
	First, we prove that $$\mbox{s.gl.}\dim A\leq\mbox{s.gl.}\dim A\otimes_kK.$$ Indeed, for each indecomposable $P$ in $\kba$, since $F$ is separable, $P$ is a direct summand of $F(P\otimes_kK)$ in $\kba$. The length of each direct summand of $P\otimes_kK$ in $\kbak$ is not larger than $\mbox{s.gl.}\dim A\otimes_kK$. As $l(F(P\otimes_kK))=l(P\otimes_kK)$, and each indecomposable direct summand does not have larger length, we have that $l(P)\leq\mbox{s.gl.}\dim A\otimes_kK$.
	
	Dually, we can prove that $$\mbox{s.gl.}\dim A\geq\mbox{s.gl.}\dim A\otimes_kK.$$ So our statement holds.
\end{proof}

Recall again that a finite-dimensional $k$-algebra $A$ is called \emph{piecewise hereditary} of type $H$ if it is derived equivalent to a hereditary abelian $k$-category $H$; $A$ is called \textit{quasi-tilted} if there is ablelian $k$-category $H$ with tilting object $T$ such that $A\simeq\mathrm{End}_H(T)^\mathrm{op}$; see \cite{HR}.

The following homological description is due to \cite{HR}.

\begin{lem}
	Let $A$ be a $k$-algebra.\\
	{\rm (1)} $A$ is piecewise hereditary if and only if $\mathrm{s.gl.}\dim A<\infty$. \\
	{\rm (2)} $A$ is quasi-tilted if and only if $\mathrm{s.gl.}\dim A\leq 2$.
\end{lem}

The above lemma and Lemma \ref{ph} immediately imply the following result.

\begin{cor}\label{phqtufe}
	Let $A$ be a $k$-algebra and $K/k$ be a finite separable field extension.\\
	{\rm (1)} The algebra $A$ is piecewise hereditary if and only if so is $A\otimes_kK$. \\
	{\rm (2)} The algebra $A$ is quasi-tilted if and only if so is $A\otimes_kK$.
\end{cor}

\subsection{Directing objects and Galois extensions} In this subsection, we prove our main theorem by using directing objects.

Let $X$ and $Y$ be two indecomposable objects in a triangulated category $\mathcal{C}$. Recall from \cite{CR} that a \textit{proper path} from $X$ to $Y$ with length $n$ is defined as a sequence of indecomposable objects $X=X_0,X_1,\dots,X_n=Y$ in $\mathcal{C}$, such that for each $i\in\{1,\dots,n\}$, either $X_i=X_{i-1}[1]$ or there is a non-zero non-isomorphism in $\mathrm{Hom}_{\mathcal{C}}(X_{i-1},X_i)$. An indecomposable object $X$ in $\mathcal{C}$ is called \textit{directing} if there is no proper paths from $X$ to $X$ with length larger than zero.

The following theorem due to \cite{CR} is the main tool we used.

\begin{thm}\cite[Corollary 5.5]{CR}\label{ph-dp}
	Let $A$ be a connected finite-dimensional $k$-algebra. Then $A$ is derived equivalent to a hereditary algebra if and only if $\dba$ contains a directing object. 
\end{thm}
\begin{prop}\label{deh}
	Let $K/k$ be a finite Galois extension and $A$ a connected $k$-algebra. If $A\otimes_kK$ is derived equivalent to a hereditary algebra, then so is $A$.
\end{prop}
\begin{proof}
	We take a connected component of $A\otimes_kK$. By Theorem \ref{ph-dp}, let $M$ be a directing object in $\dbak$. By Lemma \ref{fef} (2), there is an indecomposable object $X\in\dba$ such that $M$ and $X$ are relative. 
	
	We claim that $X$ is a directing object in $\dba$, so our statement holds by Theorem \ref{ph-dp}. If not, there is a proper path $X=X_0,X_1,\dots,X_n=X$ in $\dba$. By Lemma \ref{fef} (3), for each $i$, let $M_i$ be an indecomposable object in $\dbak$ which is relative with $X_i$. 
	
	For each non-zero non-isomorphism in $\mathrm{Hom}_{\dba}(X_{i-1},X_i)$, Lemma~\ref{bridge} implies that there is some $g_i\in G$ (the Galois group) such that there is a non-zero non-isomorphism in $$\mathrm{Hom}_{\dbak}(M,{^{g_i}N}).$$ Since $^g(-)$ is an isomorphism for each $g\in G$, there is also a non-zero non-isomorphism in $$\mathrm{Hom}_{\dbak}({^gM},{^{gg_i}N}).$$ If $X_i=X_{i-1}[1]$, we can assume that $M_i=M_{i-1}[1]={^e(M_{i-1}}[1])$, where $g_i=e$ is the unit of $G$. 
	
	So in $\dbak$, there is a proper path $$M=M_0,M'_1={^{g_1}M_1},M'_2={^{g_2g_1}M_2},\dots,M'_n={^hM},$$ where $h=g_ng_{n-1}\cdots g_1\in G$. Since $G$ is a finite group, there is a positive integer $t$ such that $h^t=e$. So from $M$ to $M$, there is a proper path $$M,M'_1,M'_2,\dots,M'_n,M'_{n+1}={^{g_1h}M}_0,\dots,M'_{2n}={^{h^2}M},\dots,M'_{tn}={^{h^t}M}=M,$$ which contradicts that $M$ is a directing object.
\end{proof}
   
\begin{thm}\label{dehufe}
	Let $K/k$ be a finite separable field extension and $A$ a $k$-algebra. Then $A$ is derived equivalent to a hereditary algebra if and only if so is $A\otimes_kK$.
\end{thm}
\begin{proof}
	If $A$ is derived equivalent to a hereditary algebra $B$, then by Lemma \ref{deot}, $A$ is equivalent to $B\otimes_kK$, which is also a hereditary algebra.
	
	Conversely, there is a finite separable field extension $L/K$ such that $L/k$ is a Galois extension. If $A\otimes_kK$ is derived equivalent to a hereditary algebra $B$, then by Lemma \ref{deot}, $A\otimes_kL$ is derived equivalent to a hereditary algebra $B\otimes_KL$. Hence (each connected component of) $A$ is derived equivalent to a hereditary algebra by Proposition \ref{deh}.
\end{proof}

According to \cite{HR}, a connected piecewise hereditary $k$-algebra is derived equivalent to either a hereditary $k$-algebra or a canonical $k$-algebra. Hence Corollary \ref{phqtufe} and the above theorem imply the following.

\begin{cor}
	Let $K/k$ be a finite separable field extension and $A$ a $k$-algebra. Then $A$ is derived equivalent to a canonical algebra if and only if so is $A\otimes_kK$.
\end{cor}

Recall that a $k$-algebra $A$ is called \textit{tilted} if there is a hereditary algebra $B$ and a tilting $B$-module $T$ such that $A\simeq\mathrm{End}_B(T)^{\mathrm{op}}$. The following result is the field extension version compared to the skew group algebra extension version proved in \cite{LZ}.

\begin{cor}
	Let $K/k$ be a finite separable field extension and $A$ a $k$-algebra. Then $A$ is tilted if and only if so is $A\otimes_kK$.
\end{cor}
\begin{proof}
	Assume that $A$ is tilted. Let $T$ be a tilting module of a hereditary $k$-algebra $B$ such that $\mathrm{End}_B(T)^{\mathrm{op}}\simeq A$. Then $T\otimes_kK$ is a tilting $B\otimes_kK$-module and $\mathrm{End}_{B\otimes_kK}(T\otimes_kK)^{\mathrm{op}}\simeq A\otimes_kK$. Since $B\otimes_kK$ is also hereditary, $A\otimes_kK$ is a tilted algebra.
	
	Conversely, assume that $A\otimes_kK$ is tilted, so it is quasi-tilted and derived equivalent to a hereditary algebra. By Corollary \ref{phqtufe}, $A$ is quasi-tilted and derived equivalent to a hereditary abelian category $H$. By Theorem \ref{dehufe}, $A$ is derived equivalent to a hereditary algebra. Hence $H$ is a module category of a hereditary algebra and $A$ is tilted.
\end{proof}

\vskip 10pt

\noindent {\bf Acknowledgments.}\quad The author is grateful to Professor Xiao-Wu Chen for his encouragement and advices. He is also grateful to the reviewers for the helpful comments.

\bibliography{}

\vskip 10pt

{\footnotesize \noindent Jie Li\\
	School of Mathematical Sciences, University of Science and Technology of China, Hefei 230026, Anhui, PR China}

\end{document}